\documentclass[a4paper,11pt]{article}

\usepackage{amsmath}
\usepackage{latexsym,amsmath}
\usepackage{amsthm}
\usepackage{amssymb}
\usepackage{enumerate}
\usepackage{soul}

\usepackage[textsize=footnotesize,color=green!40]{todonotes}
\usepackage{tikz}
\usetikzlibrary{arrows,shapes,automata,backgrounds,petri}

\usepackage{epsfig, verbatim}
\usepackage{fullpage}
\usepackage{graphicx}
\usepackage{color}

\newtheorem{theorem}{Theorem}
\newtheorem{lemma}[theorem]{Lemma}

\newtheorem{corollary}[theorem]{Corollary}

\newtheorem*{observation}{Observation}

\newtheorem{claim}{Claim}

\newtheorem{proposition}[theorem]{Proposition}

\newcommand{\cA}{\mathcal{A}}
\newcommand{\cB}{\mathcal{B}}

\newcommand{\cH}{\mathcal{H}}
\newcommand{\cR}{\mathcal{R}}

\newcommand{\sm}{\setminus}

\newcommand{\Ext}{\text{Ext}}
\renewcommand{\prec}{<_\sigma}
\renewcommand{\Pr}{\mathbb{P}}
\newcommand{\tmix}{t_\text{mix}}

\title{Frozen $(\Delta+1)$-colourings of bounded degree graphs}

\author{Marthe Bonamy\thanks{CNRS, LaBRI, Universit\'e de Bordeaux, France. Email:~\emph{marthe.bonamy@labri.fr}}
\and Nicolas Bousquet\thanks{CNRS, G-SCOP, Univ. Grenoble-Alpes, Grenoble, France. Email:~\emph{nicolas.bousquet@grenoble-inp.fr}} 
\and Guillem Perarnau\thanks{School of Mathematics, University of Birmingham. Ring Rd N, Birmingham, West Midlands B15 2TT. Email:~\emph{g.perarnau@bham.ac.uk}.}}
\date{\today}

\begin{document}

\maketitle

\begin{abstract}
Let $G$ be a graph of maximum degree $\Delta$ and $k$ be an integer.
The \emph{$k$-recolouring graph of $G$} is the graph whose vertices are $k$-colourings of $G$ and where two $k$-colourings are adjacent if they differ at exactly one vertex. It is well-known that the $k$-recolouring graph is connected for $k\geq \Delta+2$.  Feghali, Johnson and Paulusma [Journal of Graph Theory, 83(4):340--358] showed that the $(\Delta+1)$-recolouring graph is composed by a unique connected component of size at least $2$ and (possibly many) isolated vertices.

In this paper, we study the proportion of isolated vertices (also called \emph{frozen} colourings) in the $(\Delta+1)$-recolouring graph. 
Our main contribution is to show that, if $G$ is connected, the proportion of frozen colourings of $G$ is exponentially smaller than the total number of colourings. This motivates the study of the Glauber dynamics on $(\Delta+1)$-colourings. In contrast to the conjectured mixing time for $k\geq \Delta+2$ colours, we show that the mixing time of the Glauber dynamics for $(\Delta+1)$-colourings can be of quadratic order.  Finally, we prove some results about the existence of graphs with large girth and frozen colourings, and study frozen colourings in random regular graphs.
\end{abstract}

\section{Introduction}

Throughout the paper, all the graphs are labeled and simple (neither loops nor multiedges). We denote by $G=(V,E)$ a graph on the set of vertices $[n]=\{1,\dots,n\}$ and by $\Delta=\Delta(G)$ its maximum degree.  For standard definitions and notations on graphs, we refer the reader to~\cite{Diestel}. A \emph{(proper) $k$-colouring} of $G$ is a function $\alpha: V(G) \rightarrow [k]$ such that, for every edge $xy\in E$, we have $\alpha(x)\neq \alpha(y)$. Since in this paper we only consider proper colourings, we will omit the proper for brevity. The \emph{chromatic number} $\chi(G)$ of a graph $G$ is the smallest integer $k$ such that $G$ admits a $k$-colouring.

The recolouring framework consists in finding step-by-step transformations between two proper colourings such that all intermediate states are also proper. The \emph{$k$-recolouring graph of $G$}, denoted by $\mathcal{C}_k(G)$ and defined for any $k\geq \chi(G)$, is the graph whose vertices are $k$-colourings of $G$ and where two $k$-colourings are \emph{adjacent} if and only if they differ at exactly one vertex. 

Recolouring problems were introduced by the statistical physics community in the context of Glauber Dynamics to study random $k$-colourings. These problems have also attracted the interest of the theoretical computer science community, motivated by its connections to the existence of FPTAS for the number of colourings of a graph. In the last few years, the structural properties of recolouring graphs have also received quite a lot of attention. Other frameworks in which reconfiguration problems have also been studied are boolean satisfiability~\cite{BonsmaMNR14,Gopalan09}, independent sets~\cite{ItoKO14}, dominating sets~\cite{BonamyB14a,SuzukiMN14}, list edge-colouring~\cite{ItoD11}, Kempe chains~\cite{Feghali0P15a} and many others (see~\cite{Heuvel13} for a survey).

\paragraph*{Glauber dynamics.}

The study of recolouring graphs was initially motivated by problems arising from statistical physics. Glauber dynamics is a Markov chain over configurations of spin systems of graphs. This is a very general framework and ``proper $k$-colourings'' is a particular case known as the ``antiferromagnetic Potts Model at zero temperature''. 
For a graph $G=(V,E)$ with $V=[n]$ and $k\in \mathbb{N}$, \emph{Glauber dynamics for $k$-colourings} is a Markov chain $X_t$ with state space the set of $k$-colourings of $G$ and transitions defined as follows,
\begin{itemize}
\item[1)] choose $v\in [n]$ and $c\in [k]$ independently and uniformly at random;
\item[2)] for every $u\neq v$, set $X_{t+1}(u)=X_t(u)$;
\item[3)] if $c$ does not appear in the neighbours of $v$, then set $X_{t+1}(v)=c$; otherwise set $X_{t+1}(v)=X_t(v)$.
\end{itemize}
Glauber dynamics is aperiodic and reversible. If it is irreducible, then the unique stationary distribution is uniform over the set of $k$-colourings of $G$. Thus, the corresponding Markov Chain Monte Carlo algorithm provides an almost uniform sampler for colourings of $G$. Moreover, if Glauber dynamics mixes in polynomial time, it provides a Fully Polynomial Randomized Approximation Scheme (FPRAS) for the number of colourings of $G$. For a survey on this topic, we refer the interested reader to~\cite{martinelli1999lectures}.

It is easy to see that Glauber dynamics for $k$-colourings is irreducible provided that $k \geq \Delta+2$. For an aperiodic and irreducible Markov chain, we define its mixing time as the number of steps needed to reach a distribution close to stationary~(see Section~\ref{sec:prec} for a formal definition). A famous conjecture in the area is that Glauber dynamics has mixing time $O(n \log n)$, provided that $k \geq \Delta+2$.
Jerrum~\cite{jerrum1995very} and independently Salas and Sokal~\cite{salas1997absence}, showed that the Glauber Dynamics for $k$-colourings is rapidly mixing for every $k\geq 2\Delta$. Vigoda~\cite{vigoda2000improved} proved it for every $k\geq 11\Delta/6$ using a path coupling argument with a Markov chain on $k$-colourings based on Kempe recolourings. This has been recently improved to $k\geq (11/6-\epsilon)\Delta$, for some $\epsilon>0$~\cite{chen2018linear,delcourt2018rapid}. 

For $k=\Delta+1$, there exists graphs $G$ for which Glauber dynamics is not irreducible (e.g. complete graphs).  An obvious obstruction is the existence of \emph{frozen colourings}; that is, $(\Delta+1)$-colourings where every colour appears in the closed neighbourhood of every vertex. Frozen colourings correspond to fixed points in Glauber dynamics, so their existence makes the chain not irreducible. Feghali, Johnson and Paulusma~\cite{FeghaliJP14} recently showed that if $k = \Delta+1$ and $\Delta\geq 3$, frozen colourings are actually the only possible obstruction for the irreducibility of the chain. That is, the recolouring graph is composed of a unique component containing all non-frozen colourings and (possibly many) isolated frozen colourings.

This raises the following natural questions: (1) what is the size of the non-frozen component and, given it is non-empty, (2) does the Glauber dynamics rapidly mixes  there.

\paragraph*{Results}


Let $\Omega(G)$ be the set of $(\Delta+1)$-colourings of $G$ and let $\Omega^*(G)$ be the set of frozen $(\Delta+1)$-colourings of $G$. We equip $\Omega(G)$ and $\Omega^*(G)$ with the uniform distribution and understand them as probability spaces.

The main result of this paper is that frozen colourings are rare among proper colourings of connected graphs.
\begin{theorem}\label{thm:main}
Let $\Delta,n\in \mathbb{N}$, with $3\leq \Delta\leq n-2$ and let $G$ be a  connected graph on $n$ vertices and maximum degree $\Delta$. Let $\alpha$ be a colouring chosen uniformly at random from $\Omega(G)$.  Then
 \[ \Pr(\alpha \text{ is frozen})\leq (6/7)^{\frac{n}{\Delta+1}}\;. \]
\end{theorem}
Note that the condition $\Delta\leq n-2$ is necessary since every $(\Delta + 1)$-colouring of a $(\Delta+ 1)$-clique is frozen. For similar reasons, we also require $G$ to be connected. The condition of being connected can be weakened, leading to a weaker upper bound.

In the light of the result in~\cite{FeghaliJP14} and provided that $n$ is large enough with respect to $\Delta$, Theorem~\ref{thm:main} implies that the size of the component of the recolouring graph containing non-frozen colourings is always exponentially larger than the number of frozen $(\Delta+1)$-colourings. Thus, Glauber dynamics with $k=\Delta+1$ colours is likely to start in the non-frozen component, and while not irreducible, the chain will converge to the uniform distribution on it, which is an almost uniform distribution on $\Omega(G)$ (i.e. the total variation distance between it and the uniform distribution is small). From Theorem~\ref{thm:main} we obtain the following:
\begin{corollary}
Let $G$ be a graph on $n$ vertices and maximum degree $\Delta=o(n)$.
Then Glauber dynamics for $(\Delta+1)$-colourings of $G$ starting at a uniformly random vertex, gives an almost uniform sampler for random $(\Delta+1)$-colourings. 
\end{corollary}
The corollary follows, as for $\Delta=o(n)$, the probability to start in a frozen colouring is $o(1)$ and the contribution of frozen colourings to the stationary distribution is $o(1)$. The condition $\Delta=o(n)$ is probably not necessary.

The previous results motivate the question of whether the $O(n \log n)$ mixing time conjecture can be extended to Glauber dynamics for $(\Delta + 1)$-colourings that are non-frozen. Next result answers it in the negative for constant $\Delta$.
 \begin{proposition}\label{prop:tmix}
Let $\Delta,k\in \mathbb{N}$ with $\Delta\geq 3$ and $k\geq 5$. Then, there exists a $\Delta$-regular graph $G$ on $n=2k(\Delta+1)$ vertices that satisfies the following. Let $\cH$ denote the subgraph of the $(\Delta+1)$-recolouring graph induced by $\Omega(G)\setminus \Omega^*(G)$. Then, the Glauber dynamics restricted to $V(\cH)$ is ergodic and has mixing time at least
$n^2/(8(\Delta+1))$.
\end{proposition}

The previous result assumes that $\Delta\geq 3$. The case $\Delta=2$ was studied by Dyer, Goldberg and Jerrum~\cite{DyerGJ06} who proved that Glauber dynamics with $k=3$ colours on a path of order $n$ has mixing time $\Theta(n^3\log{n})$. One can check that the $3$-recolouring graph of a path is connected and has diameter $\Theta(n^2)$.  
In~\cite{FeghaliJP14}, the authors showed that the non-frozen component of the $(\Delta+1)$-recolouring graph has diameter $O(n^2)$, but, to our knowledge, no graph with superlinear diameter is known for $\Delta\geq 3$.
%

We believe that the dependence on $\Delta$ in the exponent in Theorem~\ref{thm:main} is not best possible. However, there exist graphs for which the probability a colouring is frozen is only a factor $\log{\Delta}$ in the exponent off from the upper bound in Theorem~\ref{thm:main}.  

 
%
%

%
%
%

\begin{proposition}\label{prop:lb}
For every $n_0$ and every $\Delta\geq 3$ there exist $n\geq n_0$ and a connected $\Delta$-regular graph $G$ on $n$ vertices such that if $\alpha$ is chosen uniformly at random from $\Omega(G)$, then
\[ \Pr(\alpha \text{ is frozen}) \geq e^{- \frac{3\log(\Delta)}{\Delta}\cdot n}\;. \]
\end{proposition}

The graphs in Proposition~\ref{prop:lb} contain many cliques of size $\Delta$. Our next result shows the existence of locally sparse $\Delta$-regular graphs with frozen colourings.
\begin{proposition}\label{prop:girth}
For every $\Delta,g\geq 3$ there exists a $\Delta$-regular graph $G$  with girth at least $g$ that has frozen colourings.
\end{proposition}
 



To prove this proposition we use a randomised construction based on random lifts of the $(\Delta + 1)$-clique.  Let $G_{n,\Delta}$ be a graph chosen uniformly at random among all $\Delta$-regular graphs on $n$ vertices. We conclude the paper by showing that typically, $G_{n,\Delta}$ does not admit frozen colourings.
\begin{proposition}\label{prop:whp}
For every $\Delta\geq 3$ there exist $c,n_0>0$ such that for every $n\geq n_0$, we have
$$
\Pr(G_{n,\Delta} \text{ has a frozen colouring}) \leq e^{-c n}\;.
$$

\end{proposition}

\section{Preliminaries}\label{sec:prec}

In this section we introduce some notations, definitions and observations we will use throughout the article.

Let $S_n$ the set of permutations of length $n$. We identify the set of linear orders on $[n]$ with $S_n$, and write $i\prec j$ for $\sigma(i)<\sigma(j)$. 

For every $v\in V(G)$, we define the \emph{neighbourhood} of $v$ as $N_G(v)=\{u\in V(G):\, uv\in E(G)\}$ and denote $d_G(v)=|N_G(v)|$. We also define the \emph{closed neighbourhood} of $v$ as $N_G[v]=N_G(v)\cup\{v\}$ and its \emph{second neighbourhood} as $N^2_G(v)=\{w\in V(G)\setminus N_G(v):\exists u\in N_G(v),\, uv,uw\in E(G)\}$. If the graph $G$ is clear from the context, we write $N(v)$, $d(v)$, $N[v]$ and $N^2(v)$.

\subsection{Markov chains and mixing time}
Let $\Omega$ be a finite set. For any two probability distributions $\mu$ and $\nu$ on $\Omega$, we define its \emph{total variation distance} as
$$
\|\mu-\nu\|_{TV}=\max_{A\subseteq \Omega}|\mu(A)-\nu(A)|\;.
$$

Consider an irreducible and aperiodic Markov chain on $\Omega$ with transition matrix $P$ and stationary distribution $\pi$. Define
$$
d(t)=\max_{x\in\Omega} \|P^t(x,\cdot)-\pi\|_{TV}\;,
$$
and
$$
\tmix(\epsilon)= \min \{t:\, d(t)\leq \epsilon\}\;.
$$
The \emph{mixing time} of the chain is defined as $\tmix:=\tmix(1/4)$.

\subsection{Basic properties of frozen colourings}


We now provide a number of straightforward structural properties of graphs that admit frozen colourings.

\begin{claim}\label{cla:1}
Let $G$ be a graph on $n$ vertices, maximum degree $\Delta$ and $\Omega^*(G)\neq \emptyset$. Let $\alpha\in \Omega^*(G)$. Then, the following hold:
\begin{itemize}
\item[(a)] $G$ is $\Delta$-regular;
\item[(b)] each colour class of $\alpha$ has the same size, in particular, $\Delta+1$ divides $n$;
\item[(c)] each pair of colours in $\alpha$ induces a perfect matching of $G$. 
\end{itemize}
\end{claim}
\begin{proof}
To prove (a), suppose there is a vertex $v$ with degree smaller than $\Delta$. Then $|N[v]|<\Delta+1$ and there exists $a\in [\Delta+1]$ that does not appear in $N[v]$. Thus, $v$ can be recoloured with $a$, a contradiction.

Since $G$ is $\Delta$-regular, each colour appears exactly once in each close neighbourhood. Property (b) follows from a simple double-counting argument on the size of the colour classes.

To prove (c), fix a pair of colours $a$ and $b$. Since each colour appears exactly once in every closed neighbourhood, the graph induced by colours $a$ and $b$ is $1$-regular (that is, a perfect matching).
\end{proof}

A \emph{$k$-lift of a graph $H$} is a graph $G$ with vertex set $V(G)=V(H)\times [k]$ and edge set obtained as follows: for every edge $uv\in E(H)$, we place a perfect matching between the sets $\{u\}\times [k]$ and $\{v\}\times [k]$.
\begin{claim}\label{cla:2}
Let $G$ be a graph on $n$ vertices and maximum degree $\Delta$. Then, $\Omega^*(G)\neq \emptyset$ if and only if $G$ is isomorphic to a $(n/(\Delta+1))$-lift of $K_{\Delta+1}$.
\end{claim}
\begin{proof}
Let $\alpha\in \Omega^*(G)$ and let $V_1,\dots, V_{\Delta+1}$ its colour classes. Since $\alpha$ is frozen, by Claim~\ref{cla:1}(b), for every $i\in [\Delta+1]$, $|V_i|=\frac{n}{\Delta+1}$. Since $\alpha$ is proper, each $V_i$ is an independent set. Moreover, for every $i,j\in [\Delta+1]$ with $i\neq j$, $ij\in E(K_{\Delta+1})$ and $V_i\cup V_j$ induces a perfect matching in $G$~(Claim~\ref{cla:1}(c)). Thus $G$ is isomorphic to an $(n/(\Delta+1))$-lift of $K_{\Delta+1}$.

Now, let $G$ be a graph isomorphic to a $k$-lift of $K_{\Delta+1}$. Consider $\beta\in \Omega(K_{\Delta+1})$ and define $\alpha$ by $\alpha(u,i)=\beta(u)$, for every $u\in K_{\Delta+1}$ and $i\in [k]$. By the construction of the lift, $\alpha$ is proper and frozen. Thus, $\Omega^*(G)\neq \emptyset$. 
\end{proof}

Let $G$ be a graph. A subset $X$ of vertices of $G$ is a \emph{module} if for every $v \notin X$ then either $v$ is incident to all the vertices of $X$ or $v$ is incident to none of them.

\begin{lemma}\label{lem:module}
Let $G$ be a graph with maximum degree $\Delta$ and $\Omega^*(G)\neq \emptyset$, $X$ a module of $G$ and $\alpha\in \Omega^*(G)$. Let $\beta$ be the colouring obtained from $\alpha$ by  permuting arbitrarily the colours in $X$. Then $\beta\in \Omega^*(G)$.
\end{lemma}
\begin{proof}
Let us first note that it suffices to prove Lemma~\ref{lem:module} when the permutation only swaps two colours (\emph{i.e.} is a transposition). The general result is a consequence of the fact that any permutation can be obtained through a sequence of transpositions.

Let $a$ and $b$ be two colours that appear on $X$. We denote by $X_a$ and $X_b$ the (non-empty) subsets of $X$ coloured respectively with $a$ and $b$.
Let $\beta$ be the colouring obtained from $\alpha$ by recolouring the vertices of $X_b$ with $a$ and the ones in $X_a$ with $b$. Let us prove that $\beta$ is proper and that all the colours appear in the closed neighbourhood of each vertex. This ensures that $\beta$ is frozen.
Let $v\in V \setminus X$, then either $v$ is incident to none of the vertices of $X$ or it is incident to all of them. In the first case, no colour is modified in its closed neighbourhood and the conclusion holds. In the latter case, $v$ is incident to all the vertices of $X$, so its colour in $\alpha$ is different from $a$ and $b$. Moreover, after permuting the colours, the set of colours incident to $v$ is not modified.
Now consider $v \in X$. Note that $v$ has no neighbour coloured with $a$ or $b$ in $V \setminus X$. Otherwise, since $X$ is a module, such a neighbour $w$ would be incident to the whole set $X$ which contains vertices with colours $a$ and $b$, contradicting the fact that $\alpha$ is proper. Since $\alpha$ is frozen, $v$ has two neighbours $v_a\in X_a$ and $v_b\in X_b$. After permuting the colours, $v_a$ is coloured in $b$ and $v_b$ in $a$. So all the colours still appear in $N[u]$. We conclude that $\beta\in \Omega^*(G)$.
\end{proof}

\section{Probability of frozen colourings}
%
%


A proper colouring is \emph{$1$-frugal} (or simply \emph{frugal}), if for every $v\in V(G)$ no colour appears more than once in $N_G(v)$. Denote by $\Omega^f(G)$ the set of frugal $(\Delta+1)$-colourings of $G$.
We introduce frugal colourings as they play the role of frozen colourings for graphs with maximum degree $\Delta$. In particular, if $G$ is $\Delta$-regular, then $\Omega^*(G)=\Omega^f(G)$.

Let $H=(V,E)$ be a graph with maximum degree $\Delta$ and let $\alpha\in \Omega(H)$. For every $S\subseteq V$ we denote by $\alpha_S$ the restriction of $\alpha$ onto  $S$. Observe that $\alpha_S\in \Omega(H[S])$; moreover, if $\alpha\in \Omega^f(H)$, then $\alpha_S\in \Omega^f(H[S])$. We write $\alpha(S)$ for the set of colours induced by $\alpha$ in $S$; that is, $\alpha(S)=\cup_{v\in S} \alpha(v)$.


For every $T\subseteq V$ and every $\beta \in \Omega(H[V\setminus T])$, we define the \emph{number of proper extensions of $\beta$ in $T$} as
$$
\Ext_H(T,\beta)=|\{\alpha\in \Omega(H): \,\alpha_{V\setminus T}=\beta\}|\;.
$$
and the \emph{number of frugal extensions of $\beta$ in $T$} as
$$
\Ext^f_H(T,\beta)=|\{\alpha\in \Omega^f(H):\,\alpha_{V\setminus T}=\beta\}|\;.
$$
Note that $\Ext^f(T,\beta)=0$ if $\beta\notin \Omega^f(H[V\sm T])$.

For every   $T\subseteq V$ and every $\sigma\in S_n$, we define the \emph{number of degree extensions according to $\sigma$ in $T$} as
$$
\Ext^d_H(T,\sigma):=\prod_{v\in T}( \Delta+1 -d_\sigma(v))\;,
$$
where $d_\sigma(v)$ is the number of $u\in N(v)$ such that $u\prec v$. Observe that $\Ext^d_H(T,\sigma)$ does not depend on the colouring of $V\sm T$; this will be useful in the proof. If the graph $H$ is clear from the context, we will write $\Ext$, $\Ext^f$ and $\Ext^d$.

Our next lemma relates these three types of extensions.
\begin{lemma}\label{lem:ext}
Let $T\subseteq V$ and let $\beta\in \Omega(H[V\setminus T])$. Let $\sigma\in S_n$ such that $u\prec v$ for every $u\in V\setminus T$ and $v\in T$. Then, we have
\begin{align*}
\Ext^f(T,\beta) \leq \Ext^d(T,\sigma)\leq \Ext(T,\beta)\;.
\end{align*}
\end{lemma}
\begin{proof}
The first inequality follows from the fact that every frugal colouring can obtained by extending $\beta$ greedily in the order given by $\sigma$. When assigning a colour to $v\in T$ there are two possibilities,
\begin{itemize}
\item[-] there are two coloured neighbours of $v$ that have the same colour, thus there are no frugal extensions of the current extension of $\beta$;
\item[-] there are no repeated colours in the coloured neighbourhood of $v$, thus there are exactly $\Delta+1-d_\sigma(v)$ colours available for $v$ to extend $\beta$ to a frugal colouring.
\end{itemize}

The second inequality follows from the fact that $\Ext^d(T,\sigma)$ is a lower bound for the number of greedy extensions of $\beta$ following the order $\sigma$: when we assign a colour to $v\in T$, there at most $d_\sigma(v)$ colours forbidden since some colours can be repeated in the coloured neighbourhood. The number of greedy extensions of $\beta$ following a fixed order is equal to the number of proper extensions of $\beta$.
\end{proof}

The following is the main lemma of the paper and bounds the proportion of frugal colourings of a graph.
\begin{lemma}\label{lem:main}
Let $\Delta,n\in \mathbb{N}$ with $3\leq \Delta\leq n-1$. Let $H=(V,E)$ be a graph on $n$ vertices, maximum degree $\Delta$ that contains no cliques of size $\Delta+1$. Suppose that $x\in V(H)$ satisfies $d_H(x)=\Delta$. Then, we have
\begin{align*}
\frac{|\Omega^f(H)|}{|\Omega(H)|}&\leq  \frac{6}{7}\cdot  
\frac{|\Omega^f(H[V\setminus N[x]])|}{|\Omega(H[V\setminus N[x]])|}\;.
\end{align*}
\end{lemma}

\begin{proof}

For every $\beta\in \Omega(H[V\setminus N[x]])$, 
let $S(\beta)= \beta(N_H^2(x))\subseteq [\Delta+1]$; that is, the set of colours induced by $\beta$ in the vertices at distance two from $x$.

We split the colourings of $H$ in terms of $|S(\beta)|$. For the sake of simplicity, we use $\cA:=\Omega(H)$ and $\cB:=\Omega(H[V\setminus N[x]])$. Define
\begin{align*}
\cA_1 &:= \{\alpha\in \cA:\, \alpha_{V\setminus N[x]}=\beta, |S(\beta)|\geq  \Delta/6+1 \}\\
\cA_2 &:= \cA\setminus \cA_1\;.
\end{align*}
We define these two sets similarly for each $\cA^f,\cB,\cB^f$, instead of $\cA$.

Note that to prove the lemma it suffices to show that for $i\in\{1,2\}$, we have
$$
 \frac{|\cA_i^f|}{|\cA_i|}  \leq  \frac{6}{7} \cdot \frac{|\cB^f_i|}{|\cB_i|}\;. 
$$

We split the proof into two parts:

\smallskip

\noindent \textbf{Part 1:}  
In this part, we prove that 
$$
 \frac{|\cA_1^f|}{|\cA_1|}  \leq  \frac{6}{7} \cdot \frac{|\cB^f_1|}{|\cB_1|}\;. 
$$
Fix  $\beta\in \cB_1$.  Let $\sigma$ be a linear order of $[n]$ such that
$$
v \prec x \prec y\;,
$$
for every $v\in V\setminus N[x]$ and every $y\in N(x)$.

We first compare the number of frugal extensions of $\beta$ with the number of degree extensions according to $\sigma$.

For every $c\in [\Delta+1]$, let $\beta_c$ be the colouring of $H[V\sm N(x)]$ obtained from $\beta$ by assigning colour $c$ to $x$. If $c\in S(\beta)$, then $\Ext^f(N(x),\beta_c)=0$, since for any extension $\alpha\in \cA_1$ of $\beta_c$ there exists $u\in N(x)$ such that $c$ appears at least twice in its neighbourhood.

Since $\beta\in \cB_1$, we have $|S(\beta)|\geq \Delta/6+1$. Then, by Lemma~\ref{lem:ext},
$$
\Ext^f(N[x],\beta) = \sum_{c\in [\Delta+1]} \Ext^f(N(x),\beta_c) = \sum_{c\in [\Delta+1]\setminus S(\beta)} \Ext^f(N(x),\beta_c) \leq  \frac{5\Delta}{6}\cdot \Ext^d(N(x),\sigma)\;.
$$
Let us now count the number of proper extensions of $\beta$. Since $\beta$ assigns no colour to $N(x)$, there are $\Delta+1$ choices to colour $x$. Thus, by Lemma~\ref{lem:ext},
$$
\Ext(N[x],\beta) =\sum_{c\in [\Delta+1]}  \Ext(N(x),\beta_c) \geq  (\Delta+1) \Ext^d(N(x),\sigma)
$$
If $\beta\in \cB_1\sm\cB^f_1$, then $\Ext^f(N[x],\beta)=0$ as any extension of a non-frugal colouring is non-frugal. It follows that 
\begin{align*}
|\cA^f_1| &=  \sum_{\beta\in \cB_1}  \Ext^f(N[x],\beta) = \sum_{\beta\in \cB^f_1}  \Ext^f(N[x],\beta)
\leq \frac{5\Delta}{6}\cdot  \Ext^d(N(x),\sigma)  |\cB^f_1|\;,
\end{align*}
and that
\begin{align*}
|\cA_1| &=  \sum_{\beta\in \cB_1} \Ext(N[x],\beta)
\geq  (\Delta+1) \cdot\Ext^d(N(x),\sigma)  |\cB_1|
\end{align*}
We conclude that 
\begin{align*}
\frac{|\cA^f_1|}{|\cA_1|} &\leq \frac{5\Delta/6}{\Delta+1}\cdot \frac{|\cB^f_1|}{|\cB_1|} \leq \frac{6}{7}\cdot \frac{|\cB^f_1|}{|\cB_1|}\;.
\end{align*}
\smallskip

\noindent \textbf{Part 2:}  
In this part, we prove that 
$$
 \frac{|\cA_2^f|}{|\cA_2|}  \leq  \frac{6}{7} \cdot \frac{|\cB^f_2|}{|\cB_2|}\;. 
$$

Fix $y,z\in N(x)$ such that $yz\notin E$. Such a pair exists since $d(x)=\Delta$ and $H$ contains no cliques of size $\Delta+1$. 

A vertex $w\in N[x]$ is a \emph{twin} of $x$ if and only if $N[x]=N[w]$. Let $T=\{w_1,\dots,w_t\}\subseteq N[x]$ denote the set of twins of $x$ with $w_1=x$. Note that $y,z\notin T$ and let $S=N[x]\setminus T=\{u_1,\dots,u_s\}$ with $u_1=y$ and $u_2=z$. 


Consider a linear order $\sigma$ of $[n]$ such that
$$
v \prec y \prec z \prec u_3\prec \dots\prec u_s\prec w_t\prec\dots\prec w_2\prec x\;,
$$
for every $v\in V\setminus N[x]$.

\medskip 
\noindent We split the proof into two cases:
\smallskip

\noindent \underline{\textit{Case 2.1:}}  $t\geq 17\Delta/30$.
\smallskip
%

Let $\beta\in \cB_2$.  We first bound from above the number of frugal extensions of $\beta$. For $c,c'\in [\Delta+1]$, let $\beta_{c,c'}$ be the colouring of $H[(V\setminus N[x]) \cup \{y,z\}]$ obtained from $\beta$ by assigning colour $c$ to $y$ and colour $c'$ to $z$. If $c=c'$, then $\Ext^f(N[x]\sm \{y,z\}, \beta_{c,c'}) =0$, since for any extension $\alpha\in \cA_2$ of $\beta_{c,c'}$, the colour $c$ appears at least twice in the neighbourhood of $x$. 

By Lemma~\ref{lem:ext},
$$
\Ext^f(N[x],\beta) = \sum_{c,c'\in [\Delta+1]  \atop c\neq c'}\Ext^f(N[x]\setminus\{y,z\},\beta_{c,c'}) \leq  (\Delta+1)\Delta \cdot \Ext^d(N[x]\setminus\{y,z\},\sigma)\;.
$$
Let us now count the number of non-frugal extensions of $\beta$ depending on  whether $\beta$ is frugal or not.


%

Note that if $c\in [\Delta+1]\setminus S(\beta)$, then $\beta_{c,c}$ is a proper colouring of the graph $H[(V\setminus N[x]) \cup \{y,z\}]$ but any extension of $\beta_{c,c}$ is not frugal. We can extend $\beta_{c,c}$ to a proper colouring in a greedy way following the order $\sigma$. A key observation is that when assigning a colour to $w\in T$, there will be at least $\Delta+2-d_\sigma(w)$ choices for it (instead of $\Delta+1-d_\sigma(w)$), since $y,z\in N(w)$, $y,z\prec w$ and $y,z$ have the same colour in any extension of $\beta_{c,c}$. Thus, for every $c\in [\Delta+1]\setminus S(\beta)$, we have
\begin{align*}
\Ext(N[x]\setminus\{y,z\},\beta_{c,c})
&\geq\sum_{\beta'\in\Omega(H[V\sm T])\atop \beta'_{(V\sm N[x]) \cup \{y,z\}}=\beta_{c,c}} \Ext(T,\beta') \\
&\geq\sum_{\beta'\in\Omega(H[V\sm T])\atop \beta'_{(V\sm N[x]) \cup \{y,z\}}=\beta_{c,c}} \prod_{w\in T} (\Delta+2-d_\sigma(w)) \\
&= \sum_{\beta'\in\Omega(H[V\sm T])\atop \beta'_{(V\sm N[x]) \cup \{y,z\}}=\beta_{c,c}} \!\!\!\!\!\!\!\!\Ext^d(T,\sigma)\prod_{w\in T}\left(1+\frac{1}{\Delta+1 -d_\sigma(w)}\right)\\
&\geq \Ext^d(N[x]\setminus\{y,z\},\sigma)\prod_{w\in T}\left(1+\frac{1}{\Delta+1 -d_\sigma(w)}\right)\;.
\end{align*}
Note that $w_i$ is the $i$-th largest vertex according to the linear order $\sigma$. Moreover, since $w_i$ is a twin of $x$, $d_H(w_i)=d_H(x)=\Delta$. Thus  $d_\sigma(w_i)\geq \Delta-i+1$ and it follows that
$$
\prod_{w\in T}\left(1+\frac{1}{\Delta+1 -d_\sigma(w)}\right)\geq \prod_{i=1}^t\left(1+\frac{1}{i}\right)=t+1  \geq \frac{17\Delta}{30}\;,
$$
which implies that for any $\beta\in \cB_2$,
\begin{align*}
\Ext(N[x]\setminus\{y,z\},\beta_{c,c})
&\geq  \frac{17\Delta}{30} \cdot \Ext^d(N[x]\setminus\{y,z\},\sigma)\;.
\end{align*}
If $\beta\in \cB_2^f$ and since $|S(\beta)|\leq \Delta/6$, the number of non-frozen extensions satisfies,
\begin{align*}
\Ext(N[x],\beta)-\Ext^f(N[x],\beta) &\geq  
\sum_{c\in [\Delta+1]\setminus S(\beta)}\Ext(N[x]\setminus\{y,z\},\beta_{c,c})\\
&\geq  \frac{17\Delta^2}{36} \cdot \Ext^d(N[x]\setminus\{y,z\},\sigma)\;.
\end{align*}
If $\beta\in\cB_2\setminus \cB_2^f$, then in addition to the extensions of $\beta_{c,c}$, any extension of $\beta_{c,c'}$ with $c\neq c'$ is non-frozen. Using Lemma~\ref{lem:ext}, we obtain
\begin{align*}
\Ext(N[x],\beta)-&\Ext^f(N[x],\beta) \geq  
\!\!\!\!\!\!\!\!\sum_{c,c'\in [\Delta+1]\setminus S(\beta)  \atop c\neq c'}\!\!\!\!\!\!\!\!\Ext(N[x]\setminus\{y,z\},\beta_{c,c'})+\!\!\!\!\!\!\!\sum_{c\in [\Delta+1]\setminus S(\beta)}\!\!\!\!\!\!\!\!\Ext(N[x]\setminus\{y,z\},\beta_{c,c})\\
&\geq (5\Delta/6+1) \left((5\Delta/6)\Ext^d(N[x]\setminus\{y,z\},\sigma)+(17\Delta/30) \cdot \Ext^d(N[x]\setminus\{y,z\},\sigma)\right) \\
&=  \left(\frac{5\Delta}{6}+1\right) \frac{7\Delta}{5}\cdot \Ext^d(N[x]\setminus\{y,z\},\sigma)\;.\\
\end{align*}
%
As before, it follows that 
\begin{align*}
|\cA^f_2| 
&\leq  (\Delta+1)\Delta \cdot \Ext^d(N[x]\setminus\{y,z\},\sigma)|\cB^f_2|\;,
\end{align*}
and
\begin{align*}
|\cA_2\setminus \cA^f_2| &\geq  \frac{17\Delta^2}{36} \cdot \Ext^d(N[x]\setminus\{y,z\},\sigma) |\cB_2^f|+\left(\frac{5\Delta}{6}+1\right)\frac{7\Delta}{5} \cdot \Ext^d(N[x]\setminus\{y,z\},\sigma) |\cB_2\setminus \cB_2^f|\\
\end{align*}
We conclude that 
\begin{align*} 
\frac{|\cA^f_2|}{|\cA_2|}&=
 \frac{|\cA^f_2|}{|\cA_2\setminus \cA^f_2|+|\cA^f_2|} 
\leq \frac{(\Delta+1)\Delta |\cB^f_2|}{(53/36)\Delta^2|\cB_2^f|  +(5\Delta/6+1)(7\Delta/5)|\cB_2\setminus \cB_2^f|}
\leq \frac{6}{7}\cdot \frac{|\cB^f_2|}{|\cB_2|}\;.
\end{align*}

\noindent \underline{\textit{Case 2.2:}} $t\leq 17\Delta/30$.
\smallskip

Let $\beta\in \cB_2$. Define 
\begin{align*}
\cR(\beta)&=\{\gamma\in\Omega (H[V\setminus\{x\}]):\, \gamma_{V\setminus N[x]}=\beta\},\\
\cR'(\beta)&=\{\gamma\in \cR(\beta):\,  |\gamma(N_H(x))|=\Delta\}\;.
\end{align*}
Note that $|\cR(\beta)|= \Ext_{H[V\setminus \{x\}]}(N_H(x),\beta)\geq \Ext^d_{H[V\setminus \{x\}]}(N_H(x),\sigma)= \Ext^d_H(N_H[x],\sigma)$, where the last equality holds since $d_{\sigma}(x)=\Delta$.

We will show that $|\cR'(\beta)|\leq (5/6) |\cR(\beta)|$. Note that for every $i\in [s]$, there exists at least one vertex $v_i\in N(x)$ such that $u_i v_i\notin E$; otherwise $u_i$ and $x$ would be twins. For each $\gamma\in \cR'(\beta)$ and $i\in [s]$, let $\gamma_i $ be the colouring defined as 
$$
\gamma_i(v):= \begin{cases}
    \gamma(u_i)       & \quad \text{if } v=v_i\\
      \gamma(v)       & \quad \text{otherwise.}\\
  \end{cases}
$$
If $\gamma(u_i)\notin S(\beta)$, then $\gamma_i\in \cR(\beta)\setminus \cR'(\beta)$. Since each $u_i$ has a unique colour within $N(x)$ in $\gamma$, $s\geq 13\Delta/30$ and $|S(\beta)|\leq \Delta/6$, each $\gamma$ gives rise to at least $4\Delta/15$ different colourings $\gamma_i$.

Given a colouring $\hat\gamma\in \cR(\beta)\setminus \cR'(\beta)$, let us compute how many pairs $(\gamma,i)\in \cR'(\beta)\times [s]$ are such that $\gamma_i=\hat\gamma$. If at least one such pair exists, then $\hat\gamma$ satisfies $\hat\gamma(N(x))=\Delta-1$. Since $d_H(x)=\Delta$, there is only one colour repeated in $N(x)$ and it appears exactly twice.  Since there are two colours not in $\hat\gamma (N(x))$, there are at most $4$ pairs $(\gamma,i)$ such that $\gamma_i=\hat \gamma$.

A double-counting argument over the pairs $(\gamma,\hat\gamma)$ where $\gamma\in\cR'(\beta)$ and $\hat\gamma\in \cR(\beta)\setminus \cR(\beta)$ with $\hat \gamma=\gamma_i$ for some $i\in [s]$, gives $(4\Delta/15) |\cR'(\beta)|\leq 4|\cR(\beta)\setminus \cR'(\beta)|$, implying
\begin{align*}
\frac{|\cR'(\beta)|}{|\cR(\beta)|}\leq \frac{|\cR'(\beta)|}{|\cR(\beta)\setminus \cR'(\beta)|}\leq  \frac{1}{\Delta/15 +1}=\frac{5}{6}\;,
\end{align*}
since $\Delta\geq 3$.

Now observe that if $\gamma \in \cR(\beta)\setminus \cR'(\beta)$, then 
$$
\Ext(\{x\},\gamma)\geq 2\cdot\Ext^d(\{x\},\sigma)=2\;,
$$
while if $\gamma \in \cR'(\beta)$,
$$
\Ext(\{x\},\gamma)= \Ext^d(\{x\},\sigma)=1\;.
$$
It follows that for every $\beta\in\cB_2$,
\begin{align*}
\Ext(N[x],\beta)&= \sum_{\gamma \in \cR'(\beta)} \Ext(\{x\},\gamma) + \sum_{\gamma \in \cR(\beta)\setminus \cR'(\beta)} \Ext(\{x\},\gamma)\\
&\geq  |\cR'(\beta)| + 2|\cR(\beta)\setminus \cR'(\beta)|]\\
&\geq \frac{7}{6} |\cR(\beta)|\\
&= \frac{7}{6} \cdot\Ext^d(N[x],\sigma) \;.
\end{align*}

Using Lemma~\ref{lem:ext} again, it follows that 
\begin{align*}
|\cA^f_2| 
&= \sum_{\beta\in \cB^f_2}\Ext^f(N[x],\beta) \leq \Ext^d(N[x],\sigma)|\cB^f_2|\;,
\end{align*}
and
\begin{align*}
|\cA_2| &=\sum_{\beta\in \cB_2}\Ext(N[x],\beta) \geq  \frac{7}{6}\cdot \Ext^d(N[x],\sigma) |\cB_2|
\end{align*}
Thus,
\begin{align*} 
\frac{|\cA^f_2|}{|\cA_2|}& \leq \frac{6}{7}\cdot \frac{|\cB^f_2|}{|\cB_2|}\;.
\end{align*}

\end{proof}

We will prove our main theorem by induction using Lemma~\ref{lem:main} and the following observation.
\begin{observation}
Let $G$ be a $\Delta$-regular graph of order $n$. Suppose that $\Omega^*(G)\neq \emptyset$ and fix $\alpha\in \Omega^*(G)$. For any $c\in [\Delta+1]$, the set $X=\alpha^{-1}(c)$ satisfies
\begin{itemize}
\item[-] $|X|= \frac{n}{\Delta+1}$.
\item[-] $\{N[x], x\in X\}$ is a partition of $V(G)$.
\end{itemize}
\end{observation}

\begin{proof}[Proof of Theorem~\ref{thm:main}]
We may assume that $\Omega^*(G)\neq\emptyset$, as otherwise the theorem clearly holds. Let $s\in \mathbb{N}$ such that $n=s(\Delta+1)$.  Let $X=\{x_1,\dots,x_s\}$ be the set given by the previous observation with $c=1$. Let $V_1=V(G)$ and let $H_1=G[V_1]$. For every $i\in [s]$, define $V_{i+1}= V_i\sm N_{H_{i}}[x_i]$ and $H_{i+1}=G[V_{i+1}]$. By the properties of $X$, we have $d_{H_i}(x_i)=\Delta$. Also note that $H_{i}$ contains no cliques of size $\Delta+1$.

Let $\Pr_i$ denote the uniform probability space over $\Omega(H_i)$ and let $\alpha_i$ be a uniformly random colouring of $H_i$.
For every $i\in [s]$, Lemma~\ref{lem:main} with $H=H_i$, $x=x_i$ implies
\begin{align*}
\Pr_i(\alpha_i \in \Omega^f(H_i))&= \frac{|\Omega^f(H_i)|}{|\Omega(H_i)|}
\leq   \frac{6}{7} \cdot  \frac{|\Omega^f(H_{i+1})|}{|\Omega(H_{i+1})|}
\leq   \frac{6}{7} \cdot \Pr_{i+1}(\alpha_{i+1} \in \Omega^f(H_{i+1}))\;.
\end{align*}
Since $G$ is $\Delta$-regular, a $(\Delta+1)$-colouring is frugal if and only if it is frozen.  We conclude that
\begin{align*}
\Pr(\alpha \in \Omega^*(G))&=\Pr_1(\alpha_1 \in \Omega^f(G))
\leq (6/7)^{\frac{n}{\Delta+1}}\;.\qedhere
\end{align*}

\end{proof}

 \section{Mixing time of the non-frozen component}\label{sec:tmix}

We first construct a graph $J=J(\ell)$, for $\ell\in \mathbb{N}$ (see Figure~\ref{fig:expofrozen} for an illustration). 
Let $V_1,\ldots,V_{\Delta+1}$ with $|V_i|=\ell$. The vertex set of $J(\ell)$ is $\cup_{i=1}^{\Delta+1} V_i$. For every $j\in [\Delta+1]$, denote by $v_1^j,\ldots,v_{\ell}^j$ the vertices of $V_j$.
The set of edges is $E(J)= E_1(J)\cup E_2(J)$ where
\begin{align*}
E_1(J)&= \{v_{i+1}^1 v_{i}^{\Delta+1} : i\in [\ell]\},\\
E_2(J)&= \{v_i^j v_i^{j'} : i\in [\ell],\,j,j'\in[\Delta+1],\, j \neq j'\}\setminus \{v_{i}^1 v_{i}^{\Delta+1} : i\in [\ell]\}\;.
\end{align*}
Here, addition is modulo $\ell$. Note that $J$ is $\Delta$-regular.

\begin{figure}
 \centering
 \includegraphics[scale=0.6]{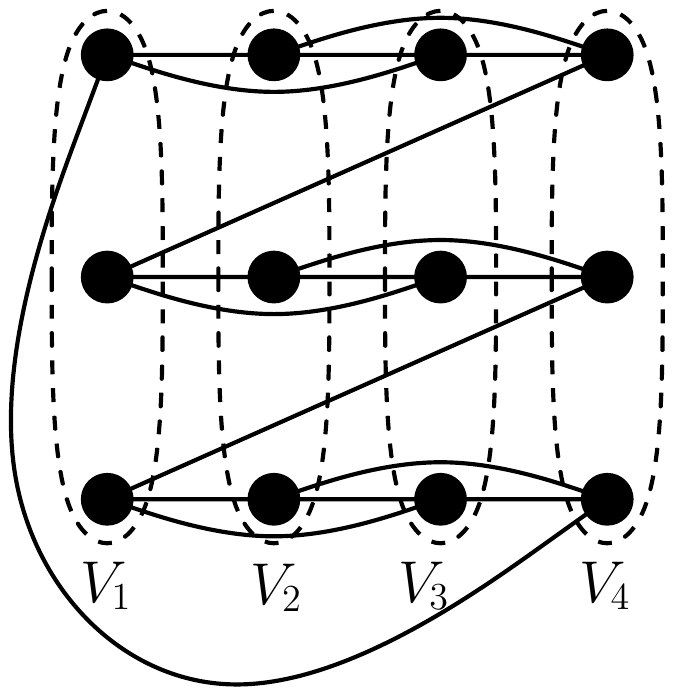}
 \caption{The graph $J(\ell)$ with $\ell=3$ and $\Delta=3$. 
 }
 \label{fig:expofrozen}
\end{figure}

We use this construction to prove Proposition~\ref{prop:tmix}.
\begin{proof}[Proof of Proposition~\ref{prop:tmix}]
Let $G=J(2k)$ be the graph defined above.
Let $\Omega_0(G)= \Omega(G)\sm \Omega^*(G)$ be the set of non-frozen colourings of $G$. The result of Feghali et al.~\cite{FeghaliJP14} implies that the Glauber dynamics with $\Delta+1$ colours on $\Omega_0(G)$ is irreducible. Since it is clearly aperiodic, the chain is ergodic and converges to the uniform distribution $\pi$ on $\Omega_0(G)$.

For every $i\in [k+1]$, define
$$
S_i=\{ \alpha \in \Omega_0(G):\, \alpha(v^1_{j})=1  \text{ for every } i+1\leq j \leq 2k+1-i\}\;.
$$
Observe that $ S_1\subseteq \dots  \subseteq S_{k+1}  = \Omega_0(G)$. Let $S_i^c=\Omega_0(G) \setminus S_i$. By symmetry, the probability that a uniformly random $\alpha$ satisfies $\alpha(v_{k+1}^1)=1$ is $1/(\Delta+1)$. Since $\Delta\geq 3$, it follows that,
\begin{align}\label{eq:small}
\pi(S^c_{k}) = 1 - \pi(S_{k})\geq 1-\frac{1}{\Delta+1}\geq \frac{3}{4}\;.
\end{align}
Consider the colouring $\beta\in \Omega_0(G)$ defined as follows: for $i\in [2k]$ and $j\in [\Delta+1]$,
$$
\beta(v^j_i)=\begin{cases}
\Delta+1 & \text{if }j=1\text{ and } i=1\\
1 & \text{if }j=\Delta+1 \text{ and }i=2k\\
j & \text{otherwise.}\\
\end{cases}
$$
Note that $\beta \in S_1$. Let $X_t$ be the Glauber dynamics with $\Delta+1$ colours on $\Omega_0(G)$ starting at $X_0=\beta$. Let $\tau_0=0$ and, for every $i\in [k]$, define the following stopping times for $X_t$,
$$
\tau_i=\min \{t:\, X_{t}\in S_{i}^c\mid X_0=\beta \}\;.
$$
From the definition of $\beta$, $\tau_i$ is also the smallest integer $t$ such that either $X_{t}(v^1_{i+1})\neq 1$ or $X_{t}(v^1_{2k+1-i})\neq 1$. 
In particular, we have
$$
0<\tau_1\leq \tau_2\leq \dots \leq\tau_{k}\;.
$$ 
%
%
%

Let $Y_i=\tau_i-\tau_{i-1}$ with $\tau_0=0$. Then $\tau_k = \sum_{i=1}^{k} Y_i$.
At time $t$, the Glauber dynamics chooses a vertex $v(t)\in [n]$ and a colour $c(t)\in [\Delta+1]$ uniformly at random, and recolours $v(t)$ with $c(t)$ if possible.
For $v\in [n]$, if $X_{t+1}(v)\neq X_t(v)$, then $v(t)=v$, and the latter occurs with probability $1/n$. Let $\{Z_i\}_{i\in [k]}$ be a collection of independent geometric random variables with parameter $2/n$. It follows that $Y_i$ stochastically dominates $Z_i$. Thus, if $Z=\sum_{i=1}^k Z_i$, there is a coupling of $\tau_k$ and $Z$ such that $Z\leq \tau_k$.
Using a concentration bound on the sum of independent geometric random variables~(see e.g. Theorem 3.1 in~\cite{jansontail}) we obtain that 
\begin{align*}
\Pr\left(Z \leq  \lambda\cdot \frac{kn}{2}\right)\leq e^{-k(\lambda-1-\log{\lambda})}\;.
\end{align*}
For $\lambda=1/2$ and since $k\geq 4$, we have $k(\lambda-1-\log{\lambda})\geq 3/4$. Thus,
\begin{align}\label{eq:1}
\Pr(\tau_k \leq k n/4) \leq \Pr(Z\leq k n/4) \leq e^{-3/4} < 1/2\;.
\end{align}
Let $\mu_t:=P^t(\beta,\cdot)$. 
Using~\eqref{eq:1}, it follows that for every $t\leq k n/4$,
\begin{align}\label{eq:2}
\mu_t(S^c_{k})=\Pr(X_t\in S^c_{k}|X_0=\beta)\leq \Pr(\tau_k \leq k n/4) + \Pr(X_t\in S^c_{k}| X_0=\beta,\,\tau_k \geq k n/4) < 1/2\;.
\end{align}
Inequalities~\eqref{eq:small} and~\eqref{eq:2} imply that for $t\leq k n/4$,
$$
d(t)=\max_{\alpha \in\Omega_0(G)} \|P^t(\alpha,\cdot) -\pi\|_{TV}\geq \|\mu_t- \pi\|_{TV}\geq  \pi(S^c_{k})- \mu_t(S^c_{k}) > \frac{1}{4}\;,
$$
and thus $\tmix=\tmix(1/4)\geq k n/4 = n^2/8(\Delta+1)$.

\end{proof}

An alternative proof to Proposition~\ref{prop:tmix} can be given using Lemma 4.2 in~\cite{hayes2005general} with $R=k$ and standard comparison arguments between continuous and discrete Markov chains. Here we provided a simpler proof that does not use the full generality of the setting in there.

\section{Graphs with many frozen colourings}

In this section we use the graph $J(\ell)$ constructed in Section~\ref{sec:tmix} to prove Proposition~\ref{prop:lb}.

\begin{proposition}\label{prop:lotoffrozen}
For every $\ell\geq 2$ and every $\Delta \geq 2$ there exists a connected graph $G$ with $\ell(\Delta+1)$ vertices such that
$$
 \Big( (\Delta-1)! \Big)^{\ell}\leq |\Omega^*(G)| \leq |\Omega(G)| \leq  \Big( 2 (\Delta+1)! \Big)^{\ell}\;.
$$
\end{proposition}
\begin{proof}

Consider the graph $G=J(\ell)$ and the colouring $\alpha_0$ of $G$ given by $\alpha_0(v_i^j)=i$ for every $i\in [\Delta+1]$ and $j\in [\ell]$. We claim that $\alpha_0\in \Omega^*(G)$. On the one hand, $\alpha_0\in \Omega(G)$ since each $V_i$ is an independent set. On the other hand, each vertex $v_i^j$ is adjacent to exactly one vertex in each set $V_k$, for $k\neq i$, thus $\alpha_0$ is frozen. 
For every $i\in [\ell]$, let $X_i=\{ v_i^2,\ldots,v_i^{\Delta} \}$. By construction of $G$, the set $X_i$ induces a clique and all the vertices of $X_i$ have the same neighbourhood in $V \setminus X_i$, \emph{i.e.} the set $X_i$ is a module of $G$. By Lemma~\ref{lem:module}, if $\alpha$ is the colouring obtained from $\alpha_0$ by permuting some colours in $X_i$, $\alpha$ is also frozen. 
For every $i \in [\ell]$, choose a permutation $\sigma_i$ of the set $\{2,\dots,\Delta\}$. Each choice of permutations gives rise to a distinct colouring obtained from $\alpha_0$ by permuting the colours in $X_i$ according to $\sigma_i$ which is  proper and frozen~(Lemma~\ref{lem:module}). It follows that $|\Omega^*(G)|\geq \left((\Delta-1)!\right)^{\ell}$.

To bound from above the number of proper $(\Delta+1)$-colourings, observe that we can create any element of $\Omega(G)$ in the following way: first, select a colouring for the clique induced by $v^1_i,v^2_i,\dots,v^\Delta_i$ (there are $(\Delta+1)!$ choices for each $i\in [\ell]$) and then select a colour for the vertices $v^{\Delta+1}_i$ (there are at most $2$ choices for each $i\in [\ell]$). It follows that $|\Omega(G)|\leq (2(\Delta+1)!)^{\ell}$.

\end{proof}
\begin{proof}[Proof of Proposition~\ref{prop:lb}]
Let $\ell_0$ be the smallest integer $\ell$ such that $\ell(\Delta+1)\geq n_0$ and let $n=\ell_0(\Delta+1)$. Let $G$ be the connected graph given by Proposition~\ref{prop:lotoffrozen} with $\ell=\ell_0$, which has $n$ vertices.
If $\alpha$ is a uniformly random colouring of $G$, we have
 \begin{align*}
   \Pr(\alpha \text{ is frozen}) 
   & = \frac{|\Omega^*(G)|}{|\Omega(G)|} 
    \leq \left(\frac{1}{2(\Delta+1)\Delta}\right)^{\frac{n}{\Delta+1}} 
    = e^{- \frac{\log(2(\Delta+1)\Delta)}{\Delta+1}\cdot n}
    \geq e^{- \frac{3\log(\Delta)}{\Delta}\cdot n}\;,
 \end{align*}
 since $2(\Delta+1)\Delta\leq \Delta^3$, for any $\Delta\geq 3$.
%
%
\end{proof}

\section{Random lifts and random regular graphs}

In this section we will prove Proposition~\ref{prop:girth} and Proposition~\ref{prop:whp}. 

To construct a graph with large girth that admits a frozen colouring, we will use random lifts of $K_{\Delta+1}$. Recall the definition of lift given in Section~\ref{sec:prec}. A random $k$-lift of $H$, is a lift where each perfect matching is chosen independently and uniformly at random among all permutations of length $k$. 

Here we will use the fact that the joint distribution of short cycles in random lifts is distributed as the product of independent Poisson random variables. Let $G$ be a random lift of $K_{\Delta+1}$, and let $C_\ell(G)$ be the number of $\ell$-cycles of $G$. Fortin and Rudinsky (Theorem~2.2 in~\cite{fortin2013asymptotic} for $K_{\Delta+1}$) show that if $Z_\ell$ are independent Poisson distributed random variables with mean $\lambda_\ell = \frac{(d+1)d(d-1)^{\ell-3}(d-3)}{2\ell}$, then, for every $m\geq 3$
$$
(C_\ell(G))^{m}_{\ell=3} \stackrel{d}{\longrightarrow} (Z_\ell)^m_{\ell=3}\;. 
$$

\begin{proof}[Proof of Proposition~\ref{prop:girth}]
Let $G$ be a random $k=(n/(\Delta+1))$-lift of $K_{\Delta+1}$. Using the Poisson approximation result, we have
$$
\Pr(g(G)\geq g) = \Pr \left(\cap_{\ell=3}^{g-1} \{C_\ell(G)=0\}\right) \longrightarrow \exp\left(- \sum_{\ell=3}^{g-1}  \lambda_\ell\right)>0\;.
$$
By Claim~\ref{cla:2}, any $k$-lift $G$ of $K_{\Delta+1}$ admits at least one frozen colouring. For sufficiently large $n$, it follows that there exists a graph with $g(G)\geq g$ and $\Omega^*(G)\neq \emptyset$.
\end{proof}

%

We have the following easy upper bound on the number of colourings of a connected graph.
\begin{claim}\label{clm:UB_cols}
Let $G$ be a connected graph on $n$ vertices and maximum degree $\Delta$, then $|\Omega(G)|\leq (\Delta+1)\Delta^{n-1}$.
\end{claim}
\begin{proof}
Consider an ordering of the vertices obtained by performing an exploration of the connected graph $G$ that starts at an arbitrary vertex and discovers new vertices by looking at the edges departing from the already discovered vertices. colour the vertices following this order; when assigning a colour to a vertex (which is not the starting one) at least one of its neighbours has been already coloured. Thus, at least one of the $\Delta+1$ colours is forbidden for these vertices.
\end{proof}

Bender and Canfield~\cite{bender1974asymptotic} proved that for any fixed $\Delta$, the number of $\Delta$-regular graphs as $n\to \infty$ is
$$
r(n,\Delta)=(1+o(1))\sqrt{2} e^{(1-\Delta^2)/4}\left(\frac{n^\Delta \Delta^\Delta}{e^\Delta (\Delta!)^2}\right)^{n/2}\sim C(\Delta)\left(\frac{e\cdot n}{\Delta}\right)^{\Delta n/2}(2\pi\Delta)^{-n/2} \;.
$$

%
\begin{proof}[Proof of Proposition~\ref{prop:whp}]
We count the number of pairs $(G,\alpha)$, where $G$ is $\Delta$-regular and $\alpha\in \Omega^*(G)$. By Claim~\ref{cla:2}, if $G$ has a frozen colouring, then it is a $k=(n/(\Delta+1))$-lift of $K_{\Delta+1}$. Moreover, every frozen colouring gives an embedding of $G$ as a lift of $K_{\Delta+1}$. Thus, the number of pairs $(G,\alpha)$ is the number of $k$-lifts of $K_{\Delta+1}$. There are $\binom{n}{k,\ldots,k}$ ways to split the vertex set into $\Delta+1$ parts of size $k$, and $(k!)^{\binom{\Delta+1}{2}}$ ways to select the perfect matchings between them. Therefore, the number of such pairs is
\begin{align*}
l(n,\Delta)&=\binom{n}{k,\ldots,k}\left(k!\right)^{\binom{\Delta+1}{2}}
= n^{O(1)}\left(\frac{n}{(\Delta+1)e}\right)^{\Delta n/2}(\Delta+1)^{n} \;.
\end{align*}
Since every graph that admits a frozen colouring contributes in at least one in the number of such pairs,
\begin{align*}
\Pr(\Omega^*(G_{n,\Delta})\neq \emptyset)&\leq \frac{l(n,\Delta)}{r(n,\Delta)}\\
&= n^{O(1)}e^{-\Delta n}\left(1+\frac{1}{\Delta}\right)^{-\Delta n/2} (2\pi (\Delta+1)^3)^{n/2}\\
&= n^{O(1)}\exp(-f(\Delta) n/2)\;,
\end{align*}
where $f(x)=2x+x\log{(1+1/x)}-3\log{(x+1)}-\log{(2\pi)}$. It can be checked that $f(x)>0$ for $x\geq 3$.
%

\end{proof}

\bibliography{recoloring}

\bibliographystyle{plain}

\end{document}